\documentclass[11pt,letterpaper,reqno]{amsart}

\usepackage[margin=1.02in]{geometry}
\usepackage[T1]{fontenc}
\usepackage[utf8]{inputenc}
\usepackage{lmodern}
\usepackage{microtype}
\usepackage{tikz-cd}

\usepackage{amssymb,mathtools,mathrsfs}
\usepackage{graphicx,booktabs,enumitem}

\usepackage{xcolor}
\usepackage{verbatim}
\definecolor{paperblue}{HTML}{124878}
\usepackage[
  colorlinks=true,
  linkcolor=paperblue,
  citecolor=paperblue,
  urlcolor=paperblue,
  linktoc=all,
  bookmarksnumbered=true,
  pdfencoding=auto,
  psdextra
]{hyperref}
\hypersetup{pdftitle={Dynamics of meromorphic maps on compact K\"ahler manifolds},
  pdfauthor={Muhan Luo and Qi Zhou}}

\usepackage[
  backend=biber,
  style=alphabetic,
  maxnames=10,
  giveninits=true,
  sortcites=true,      
  sorting=nyt          
]{biblatex}
\renewbibmacro*{issue+date}{%
  \ifboolexpr{test {\iffieldundef{issue}} and test {\iffieldundef{year}}}
    {}
    {\printtext[parens]{\printfield{issue}\setunit*{\addspace}\usebibmacro{date}}}%
  \newunit}

\usepackage{aliascnt}
\usepackage[capitalize,nameinlink,noabbrev]{cleveref}

\newcommand{\C}{\mathbb{C}}

\newcommand{\Z}{\mathbb{Z}}

\newcommand{\Pb}{\mathbb{P}}

\DeclareMathOperator{\Fix}{Fix}

\DeclareMathOperator{\supp}{supp}
\DeclareMathOperator{\dist}{dist}

\DeclareMathOperator{\diam}{diam}

\renewcommand{\d}{\mathop{}\!\mathrm{d}}

\newcommand{\cC}{\mathcal{C}}

\newcommand{\cG}{\mathcal{G}}

\newcommand{\cS}{\mathcal{S}}

\newcommand{\cV}{\mathcal{V}}

\newcommand{\vep}{\varepsilon}

\newcommand{\newsharedtheorem}[2]{%
  \newaliascnt{#1}{theorem}%
  \newtheorem{#1}[#1]{#2}%
  \aliascntresetthe{#1}%
}

\theoremstyle{plain}
\newtheorem{theorem}{Theorem}[section]
\newsharedtheorem{lemma}{Lemma}
\newsharedtheorem{proposition}{Proposition}
\newsharedtheorem{corollary}{Corollary}

\theoremstyle{definition}
\newsharedtheorem{definition}{Definition}
\newsharedtheorem{example}{Example}

\theoremstyle{remark}
\newsharedtheorem{remark}{Remark}

\Crefname{theorem}{Theorem}{Theorems}
\Crefname{lemma}{Lemma}{Lemmas}
\Crefname{proposition}{Proposition}{Propositions}
\Crefname{corollary}{Corollary}{Corollaries}
\Crefname{definition}{Definition}{Definitions}
\Crefname{example}{Example}{Examples}
\Crefname{remark}{Remark}{Remarks}
\Crefname{appendix}{Appendix}{Appendices}

\numberwithin{equation}{section}

\title[MME for meromorphic maps on compact K\"ahler manifolds]{Measures of maximal entropy for meromorphic maps on compact K\"ahler manifolds}
\author{Muhan Luo}
\address{Department of Mathematics, National University of Singapore, Singapore 119076}
\email{luomuh98@gmail.com}

\author{Qi Zhou}
\address{Department of Mathematics, National University of Singapore, Singapore 119076}
\email{e1124859@u.nus.edu}

\date{\today} 

\begin{document}

\begin{abstract}
Let $f$ be a dominant meromorphic self-map of a compact K\"ahler manifold. Assume that $f$ has a strictly dominant intermediate dynamical degree $\lambda$ and that the number of isolated $n$-periodic points, counted with multiplicity, is at most $\lambda^n+o(\lambda^n)$. Then $f$ admits at most one adapted invariant probability measure of entropy $\log\lambda$. If such a measure exists, its natural extension is Bernoulli, and the isolated periodic points equidistribute towards it when counted with or without multiplicity. We apply this criterion to several classes of maps in complex dynamics, resolving longstanding open problems.
\end{abstract}

\maketitle
\tableofcontents

\section{Introduction}\label{sec:intro}
Entropy is one of the fundamental invariants of a dynamical system, quantifying the complexity of its orbit structure. Measures of maximal entropy (MME) maximize this complexity. Understanding their existence, uniqueness, statistical properties and relation to periodic orbits consequently occupies a central place in ergodic theory.

Extensive work has been done on these questions in various holomorphic or meromorphic complex dynamical systems, where the construction and properties of the Green currents play a central role; see for instance \cite{DS05JAMS,Bayraktar13JGA}. Following foundational work on polynomial automorphisms of $\mathbb C^2$ \cite{BS91,BS92,BLS93a}, MMEs have been constructed and studied for automorphisms of compact K\"ahler manifolds \cite{Cantat01Acta,DS05JAMS,DS10JAG,BD26AM}, meromorphic maps with dominant topological degree \cite{Guedj05Annals,DS06CMH}, surface maps with small topological degree \cite{BD05Duke,DDG1Indiana,DDG2CMH,DDG3ENS}, and higher-dimensional rational maps \cite{DS05JFA,dTH-Vigny,Vigny14AIF}. See also the surveys \cite{S99PS,Guedj10,ABeBrDScSi10Book}. Given the extensive literature on these topics, we do not attempt an exhaustive survey. Further references are provided where relevant throughout the paper.

Classical proofs of uniqueness of MME and equidistribution of periodic points typically rely on delicate geometric analysis, including control of local inverse branches or of the size and geometry of stable and unstable manifolds. Obtaining such control becomes particularly challenging in higher dimensions. Uniqueness has been established for holomorphic endomorphisms of projective space \cite{BD01IHES}, meromorphic maps with dominant topological degree \cite{DS06CMH}, regular polynomial automorphisms \cite{dTh10}, and holomorphic automorphisms of compact K\"ahler manifolds \cite{dTH-Dinh12AM}.

For periodic-point equidistribution, a common strategy combines lower bounds with sharp upper bounds on the total number of periodic points. The lower bounds are obtained by constructing sufficiently many periodic points with the expected distribution, using Pesin theory or contracting local inverse branches. Previous works include \cite{Lyubich83FAA,BLS93b,DDG3ENS,Briend-Duval99,DNT15,Dujardin06duke,DS16MAnn,LZ25}. See also \cite{DDK25,TDK26}. Extending these approaches to more general meromorphic maps nevertheless remains challenging, as the required control of invariant manifolds, Green currents and their intersections often depends strongly on the geometry of the maps and the ambient spaces.

In this paper, we develop a new unified approach to both questions. The central idea is to transfer them from complex dynamics to countable Markov shifts through symbolic coding. We combine the coding theorem of Lima--Obata--Poletti \cite{LOP26JEMS}, applied to the natural extension, with Buzzi's injective recoding theorem \cite{Buzzi20JMD}. The resulting coding is injective on periodic points, allowing an upper bound on the number of isolated periodic points of the original map to pass to the symbolic system. In contrast to the situation in complex dynamics, the desired lower bounds are more accessible for countable Markov shifts; see \cref{lemma:lower bound}. Our approach brings the upper bounds available from complex geometry into this symbolic setting. This yields strong dynamical conclusions at the symbolic level, including uniqueness of MME and equidistribution of periodic points, which can then be transferred back to the original complex dynamical system.

We now introduce the necessary notions to state our main theorem. Let $X$ be a compact connected K\"ahler manifold of complex dimension $k$, equipped with a K\"ahler form $\omega$, and let $f\colon X\dashrightarrow X$ be a dominant meromorphic map. Its dynamical degrees are
\[
    d_j(f):=\lim_{n\to\infty}\left(\int_X(f^n)^*\omega^j\wedge\omega^{k-j}\right)^{1/n}, \qquad 0\leq j\leq k.
\]
In particular, $d_0(f)=1$ and $d_k(f)$ is the topological degree. We refer to \cite{DS05Annals} for their basic properties. Denote the indeterminacy set by $I(f)$ and the critical locus in $X\setminus I(f)$ by $C(f)$, and set
\[
    \cS:=I(f)\cup C(f),\qquad \rho(x):=\min\{1,\dist(x,\cS)\}.
\]
The set $\cS$ is a proper analytic subset of $X$. When $\cS=\varnothing$, we set $\rho\equiv1$.

A Borel probability measure $\nu$ is $f$-invariant if $\nu(I(f))=0$ and $f_*\nu=\nu$. Following \cite{ALP24MAMS,LOP26JEMS}, such a measure is called \emph{adapted} if
\[
    \int_X-\log\rho(x)\d\nu(x)<\infty.
\]
In particular, $\nu(\cS)=0$. For an adapted invariant measure $\nu$, the \emph{natural extension} of the system $(X,f,\nu)$ consists of regular orbits $(x_j)_{j\in\Z}$ satisfying $f(x_j)=x_{j+1}$, with the invertible shift $\widehat f$ and the unique invariant lift probability measure $\widehat\nu$. Its construction, metric and entropy properties are described in \cref{sec:natural-extension}. An invertible measure-preserving system is called \emph{Bernoulli} if it is measurably isomorphic, modulo null sets, to a two-sided shift with a product probability measure. We call a point $x$ \emph{$n$-periodic} if its forward orbit avoids $I(f)$ and $f^n(x)=x$. Let $P_n(f)$ denote the isolated $n$-periodic points, counted with multiplicity. 

We call $f$ \emph{cohomologically $p$-hyperbolic} if $d_p(f)>\max_{j\ne p}d_j(f)$. Since $d_0(f)=1$ and $d_k(f)\geq1$, this necessarily implies $p\geq1$ and $d_p(f)>1$. The expanding case $p=k$, in which the topological degree is strictly dominant, has been extensively studied; see \cite{DNT15} and the references therein. We therefore focus on the intermediate case $1\leq p\leq k-1$, although most of our arguments also apply when $p=k$.

Guedj conjectured that a cohomologically hyperbolic meromorphic map admits a canonical measure which is the unique measure of maximal entropy, is mixing and hyperbolic, and describes the asymptotic distribution of saddle periodic points; see \cite[Conjecture~3.2]{Guedj05ETDS} and \cite[p.~100]{Guedj10}. Our main result is the following criterion, which gives a partial answer to this conjecture.

\begin{theorem}\label{thm:main}
Let $f\colon X\dashrightarrow X$ be a dominant meromorphic map of a compact connected K\"ahler manifold of complex dimension $k$. Suppose that $\lambda:=d_p(f)>\max_{j\ne p}d_j(f)$ for some $1\le p\le k-1$, and assume that
\begin{equation}\label{eq:criterion-upper}
    \#P_n(f)\leq\lambda^n+o(\lambda^n).
\end{equation}
Then the following statements hold.
\begin{enumerate}[label=\textup{(\arabic*)}]
    \item There is at most one adapted invariant probability measure $\mu$ with maximal entropy $h_\mu(f)=\log\lambda$.

    \item If such a measure $\mu$ exists, its natural extension is Bernoulli. If moreover $f$ is bimeromorphic, $(X,\mu,f)$ itself is Bernoulli.

   \item If such a measure $\mu$ exists, there is a constant $\chi>0$ with the following property. For each $n\geq1$, let $P_{n,\chi}(f)\subset P_n(f)$ consist of the points whose orbits avoid $\cS$ and for which $D_xf^n$ has exactly $p$ eigenvalues of modulus greater than $e^{n\chi}$ and $k-p$ eigenvalues of modulus less than $e^{-n\chi}$. Every point in $P_{n,\chi}(f)$ is a simple fixed point of $f^n$, and hence has multiplicity one.
   Then, for any sequence of subsets $(Q_n)_{n\geq 1}$ satisfying $P_{n,\chi}(f)\subset Q_n\subset P_n(f)$, we have
    \begin{equation}\label{eq:equidist}
        \lambda^{-n}\sum_{x\in Q_n}\delta_x\to\mu
    \end{equation}
    as $n$ goes to infinity where $\delta_x$ denotes the Dirac measure at $x$.
\end{enumerate}
\end{theorem}

\begin{remark}\label{rmk:sharper-multipliers}
    In our setting, de Th\'elin \cite{deThelin08Invent} shows that $\mu$ as above has exactly $p$ positive and $k-p$ negative Lyapunov exponents. The multiplier bounds in (3) admit a refinement in terms of the Lyapunov exponents of $\mu$. Set $\chi^+:=\chi_p(\mu)>0$ and $\chi^-:=\chi_{p+1}(\mu)<0$. For $0<\vep<\min\{\chi^+,-\chi^-\}$, let $P_{n,\vep}(f)$ consist of the periodic points in $P_n(f)$ with $p$ multipliers of modulus greater than $e^{n(\chi^+-\vep)}$ and $k-p$ multipliers of modulus less than $e^{n(\chi^-+\vep)}$. Then \eqref{eq:equidist} also holds for every sequence $P_{n,\vep}(f)\subset Q_n\subset P_n(f)$. We omit the proof of this refinement for brevity.
\end{remark}

\begin{remark}
    For holomorphic automorphisms with simple action on cohomology, our criterion recovers the equidistribution result established in \cite{LZ26} within a more general framework. The uniqueness of MME in this setting was previously established by de Thélin–Dinh \cite{dTH-Dinh12AM}.
\end{remark}

The arguments in the proof of \cref{thm:main}, together with the latest work of Hu \cite{Hu26}, also give the following finiteness result for dominant meromorphic maps, only assuming cohomological hyperbolicity.

\begin{proposition}\label{prop:finite MMEs}
    Let $f\colon X\dashrightarrow X$ be a dominant meromorphic map of a compact connected K\"ahler manifold of complex dimension $k$. Suppose that $\lambda:=d_p(f)>\max_{j\ne p}d_j(f)$ for some $1\le p\le k-1$. Then $\# P_n(f)=O(\lambda^n)$, and there are only finitely many adapted ergodic invariant probability measures $\nu$ satisfying $h_\nu(f)=\log\lambda$.
\end{proposition}

The following proposition establishes the upper bound \eqref{eq:criterion-upper} in each case, allowing us to apply \cref{thm:main}. Further background on these maps is given in \cref{sec:applications}. Recall that, for $0\leq p\leq k$, we call $f$ \textit{algebraically $p$-stable} if $(f^n)^*=(f^*)^n$ on $H^{p,p}(X,\C)$ for every $n\geq 1$.

\begin{proposition}\label{prop:applications}
Let $f\colon X\dashrightarrow X$ be a dominant meromorphic map of a compact connected K\"ahler manifold of complex dimension $k$. Suppose that $\lambda:=d_p(f)>\max_{j\ne p}d_j(f)$ for some $1\le p\le k-1$, and that one of the following conditions holds:
\begin{enumerate}[label=\textup{(\arabic*)}]
    \item $k=2$ and $f$ is algebraically $1$-stable;
    \item $X=\Pb^k$ and $f$ is algebraically $p$-stable.
\end{enumerate}
Then the isolated periodic points, counted with multiplicity, satisfy $\#P_n(f)\leq\lambda^n+o(\lambda^n)$. Consequently, all the conclusions of \cref{thm:main} hold for $f$.
\end{proposition}

The known constructions of MMEs in these settings, together with the precise hypotheses under which our criterion applies, are discussed in \cref{sec:applications}. Below we explain how \cref{prop:applications} relates to earlier results and addresses several longstanding open problems.
\medskip


\noindent\textbullet For the meromorphic surface maps considered by Diller--Dujardin--Guedj \cite{DDG3ENS}, the new conclusion is uniqueness among adapted invariant MMEs. Under their hypotheses \textup{(H1)--(H3)}, this answers the question in \cite[p.~239]{DDG3ENS}; see also \cite[Section~1.4]{Dujardin22}. The Bernoulli property of the natural extension was obtained in that work. Moreover, their Theorem~C, together with \cite[Corollary~1.3]{DNT17BLMS}, gives equidistribution of all saddle periodic points. Our argument gives a new proof for these conclusions with a stronger statement in \cref{thm:main}-(3).
\medskip

\noindent\textbullet Consider the birational maps of $\Pb^k$ studied by de Th\'elin--Vigny \cite{dTH-Vigny}. Their construction of equilibrium measures extends the work of Bedford--Diller on birational surface maps \cite{BD05Duke} to higher dimensions. Our equidistribution of saddle periodic points answers the question raised in \cite[Remark~1.0.1]{dTH-Vigny}. In dimension two, this equidistribution was already established by Dujardin under the Bedford--Diller condition \cite{Dujardin06duke}. Our uniqueness result answers the question raised in \cite[Remark~3.4.20]{dTH-Vigny} within the class of adapted invariant probability measures.

\medskip
\cref{Markov} is a brief introduction to countable Markov shift and its basic properties. \cref{sec:symbolic} introduces the natural extension and the main results from symbolic dynamics. \cref{sec:proof-main} proves \cref{thm:main}. The applications to some specific classes of maps are proved in \cref{sec:applications}. The assumptions needed for \cref{thm:mme-coding} are verified in \hyperref[app:estimates]{Appendix}.

\medskip
\noindent\textbf{Acknowledgement.} The second author is supported by the President's Graduate Fellowship.
\medskip

\noindent\textbf{Declaration of AI use.} During the preparation of \cite{LZ26}, whose main idea was suggested by GPT-6 Astra, the authors observed that the same method could also be used to prove uniqueness of measures of maximal entropy, and that the results of \cite{LOP26JEMS,ALP24MAMS} allow the approach to extend beyond the holomorphic setting to meromorphic maps, leading to the present paper. The authors prepared the manuscript with AI assistance in language editing and take full responsibility for all mathematical content.

\section{Countable Markov shifts}\label{Markov}

\subsection{Basic notions}

We recall basic results of countable Markov shifts. All results mentioned here can be found in \cite[Chapter 7]{Kitchens} and \cite[Chapter 1]{Durrett99ESP}. Let $\cG=(\cV, E)$ be a directed graph with a countable collection of vertices $\cV$ such that every vertex has at least one ingoing and one outgoing edge in $E$. We always assume that there is at most one directed edge from \(u\) to \(v\) for every ordered pair of vertices \((u,v)\). We write $u\to v$ when such an edge exists. The \textit{adjacency matrix} is defined to be the countable matrix $A$ such that the vertices of the graph are the indices for the rows and columns of $A$, and for any two vertices $u,v$, 
\[
    A_{uv}=
    \begin{cases}
        1, & u\to v;\\
        0, & \text{otherwise}.
    \end{cases}
\]
The \textit{countable Markov shift} associated to $\cG$ is the set
 \[
    \Sigma=\{(v_i)_{i\in\Z}\in\cV^\Z: A_{v_iv_{i+1}}=1,\ \forall i\in\Z\}.
\]
equipped with the left-shift
\[
    \sigma \colon \Sigma \to \Sigma, \qquad \sigma\bigl((v_i)_{i\in\mathbb{Z}}\bigr)= (v_{i+1})_{i\in\mathbb{Z}},
\]
and the metric
\[
    d(u,v):= \exp\left(-\min\left\{n\in\mathbb{N}_0 :u_n\neq v_n \text{ or } u_{-n}\neq v_{-n}\right\}\right)
\]
for \(u=(u_n)_{n\in\mathbb{Z}}\), \(v=(v_n)_{n\in\mathbb{Z}}\). With this metric \(\Sigma\) is a complete separable metric space, i.e., a metric Polish space. 

A \textit{path} of length $n$ is a sequence of vertices $(u_0,u_1,...,u_n)$ such that $u_i\to u_{i+1}$ for all $0\leq i\leq n-1$. We will write $u_0\to u_1\to \cdots \to u_n $ for such a path. This path is called a \textit{loop} if $u_0 = u_n$. We can consider the product of $A$ with itself by the usual law of matrix product. In particular, $(A^n)_{ab}$ is the number of paths $(u_0,u_1,...,u_n)$ such that $u_0=a$ and $u_n=b$. By convention, $(A^n)_{ab}$ takes $\infty$ if there are infinitely many paths of length $n$ from $a$ to $b$. For integers $m,n$ with $n\geq 0$, and an admissible path $u_0\to u_1\to \cdots \to u_n $, define
\[
    [u_0,\ldots,u_n]_m:=\{x\in\Sigma:x_j=u_{j-m}\text{ for }m\leq j \leq n+m \}.
\]

\begin{definition}
Two vertices $u,v\in\mathcal V$ \emph{communicate} if there are paths from $u$ to $v$ and from $v$ to $u$, allowing paths of length zero. The equivalence classes of this relation are called the \emph{strongly connected components} of $\mathcal G$. The Markov shift $\Sigma$ is \emph{irreducible} if $\mathcal V$ is a single strongly connected component and contains a loop of positive length; equivalently, for every $u,v\in\mathcal V$, there exists $n\geq 1$ such that $(A^n)_{uv}>0$. 
\end{definition}

A \emph{Markov subshift} of $\Sigma=\Sigma(\mathcal G)$ is a shift $\Sigma(\mathcal G')$ associated with a subgraph $\mathcal G'\subset\mathcal G$ satisfying the same assumptions as $\mathcal G$. 

\begin{definition}
The \emph{period} of an irreducible Markov shift $\Sigma$ is
\[
    q:=\gcd\{n\geq1:(A^n)_{uu}>0\},
\]
where $u\in\mathcal V$ is arbitrary. This value is independent of the choice of $u$. The shift $\Sigma$ is \emph{aperiodic} if $q=1$.
\end{definition}

Every ergodic shift-invariant probability measure on $\Sigma$ is supported on an irreducible Markov subshift, as shown by the following lemma. This is already mentioned in the works \cite{Sarig13JAMS, Katok07JMD,Buzzi09ETDS,Ovadia18JMD}. For a complete proof, see for instance \cite{LZ26}.

\begin{lemma}\label{lem:ergodic-component}
Let $(\Sigma,\sigma)$ be a countable Markov shift with alphabet $\mathcal V$, and let $\nu$ be an ergodic $\sigma$-invariant Borel probability measure. Then there is a strongly connected component $\mathcal C$ containing a loop of positive length such that
\[
    \supp(\nu)\subset\Sigma({\mathcal C}),
\]
where $\Sigma(\mathcal{C})$ is the irreducible Markov subshift associated with $\mathcal{C}$. Moreover, every symbol $a$ with $\nu([a]_0)>0$ belongs to $\mathcal C$.
\end{lemma}

Following Sarig \cite{Sarig13JAMS}, we define the \emph{regular part} of a countable Markov shift $\Sigma$ by
\begin{equation}\label{eq:regular}
    \Sigma^\#:=\left\{(v_i)_{i\in\mathbb{Z}}\in\Sigma : \exists\,v',w'\in\mathcal{V}(\Sigma),\ \exists\,n_k,m_k\uparrow\infty\ \text{s.t.}\  v_{n_k}=v', \  v_{-m_k}=w',\  \forall\,k\geq 1 \right\}.
\end{equation}
By the Poincaré recurrence theorem, every \(\sigma\)-invariant probability measure $\mu$ is carried by \(\Sigma^\#\), i.e., $\mu(\Sigma^{\#})=1$. Furthermore, every periodic point of $\sigma$ is in $\Sigma^\#$. Note that $\Sigma^\#$ is not necessarily a closed subset of $\Sigma$.

\subsection{Positive recurrence and entropy}\label{subsec:recurrent}

\begin{definition}
    For an irreducible countable Markov shift $\Sigma$, the \textit{Perron value} of $\Sigma$ is defined to be
\[
    \lambda:={\limsup}_{n\to\infty}((A^n)_{uu})^{1/n},
\]
for any $u\in\cV$. It is independent of the choice of $u$ and takes values in $[1, \infty]$.
\end{definition}

We now discuss recurrence for an irreducible countable Markov shift $\Sigma$ with adjacency matrix $A$ and finite Perron value $\lambda$. Then $(A^n)_{uv}<\infty$ for all vertices $u,v$ and all $n\geq 1$. With the convention $A^0=I$, for $n\geq 0$, define $a_{uv}(n):=(A^n)_{uv}$. For $n\geq 1$, define $l_{uv}(0)=0 $, $l_{uv}(1)=A_{uv}$, and $l_{uv}(n+1):=\sum_{w\neq u} l_{uw}(n)A_{wv}$. The coefficient $l_{uv}(n)$ counts the paths from $u$ to $v$ of length $n$ that do not return to $u$ before time $n$.

\begin{definition}
     An irreducible countable Markov shift $\Sigma$ with finite Perron value $\lambda$ is \textit{recurrent} if for some $u\in\cV$, $\sum_{n=1}^{\infty} a_{uu}(n)/\lambda^n= \infty$. It is \textit{positive recurrent} if it is recurrent and $\sum_{n=1}^\infty nl_{uu}(n)/\lambda^n<\infty$. Both definitions are independent of the choice of $u$.
\end{definition}

The following is the Perron-Frobenius theorem for countable Markov shifts. For reference, see \cite[Theorem 7.1.3]{Kitchens}.
\begin{lemma}\label{Perron-Frobenius}
    Let $\Sigma$ be an irreducible countable Markov shift with finite Perron value $\lambda$. If $\Sigma$ is recurrent, $\lambda$ has strictly positive left and right eigenvectors $r=(r_u)$ and $\ell=(\ell_u)$ with $rA=\lambda r$ and $A\ell=\lambda \ell$. And the eigenvectors are unique up to constant multiples. Moreover, if $\Sigma$ is positive recurrent, we have that $r\cdot \ell=\sum_u r_u\cdot \ell_u<\infty$.
\end{lemma}

The \textit{Gurevich entropy} of the Markov shift $\Sigma$ associated with $\cG$ is given by
\[
     h(\Sigma):=\sup\{h_{top}(\Sigma(\cG'),\sigma):\cG' \text{ is a finite subgraph}\}.
\]
The variational principle is still valid for Markov shifts:
\begin{theorem}[\cite{Gurevich69}]\label{variational}
    For an irreducible countable Markov shift $\Sigma$ with Perron value $\lambda$, we have 
    \[
         h(\Sigma)=\sup_{\mu} h_\mu(\sigma)=\log \lambda,
    \]
    where $\mu$ runs over all invariant Borel probability measures. If any of the three quantities above is infinite, then so are the other two.
\end{theorem}

The next theorem gives a necessary and sufficient condition for the existence of a measure of maximal entropy when the Markov shift is irreducible.

\begin{theorem}[\cite{Gurevich70}]\label{thm:MME-posi recurr}
     Let $\Sigma$ be an irreducible countable Markov shift with Perron value $\lambda< \infty$. Then, it admits a measure of maximal entropy if and only if $\Sigma$ is positive recurrent and such a measure is unique if it exists.
\end{theorem}

We will use the following consequence of \cite[Lemma~6.2]{Buzzi20JMD}, which gives a lower bound for the number of periodic points of an irreducible, positive recurrent Markov shift. The original result counts points of minimal period $n$, but the weaker formulation below suffices for our purposes.

\begin{lemma}\label{lemma:lower bound}
    Let $(\Sigma, \sigma)$ be an irreducible countable Markov shift which is positive recurrent with Gurevich entropy $h$ and period $q$. Then 
    \[
       \liminf_{\substack{n\to\infty,\\  q|n}} e^{-nh}\#\Fix(\sigma^n)\geq q.
    \]
\end{lemma}

\subsection{Markov chains}

Let $(\Sigma,\sigma)$ be an irreducible countable Markov shift with $0<h(\Sigma)<\infty$ and a measure of maximal entropy $\mu$. Set $\lambda:=e^{h(\Sigma)}$. By \cref{thm:MME-posi recurr}, $\Sigma$ is positive recurrent. Choose the Perron--Frobenius eigenvectors $r$ and $\ell$ in \cref{Perron-Frobenius}, normalized by $r\cdot\ell=1$, and define
\begin{equation}\label{eq:transition-stationary}
    K_{uv}:=\frac{A_{uv}\ell_v}{\lambda\ell_u},\qquad \rho_u:=r_u\ell_u.
\end{equation}
The eigenvector relations give $\sum_v K_{uv}=1$, $\sum_u\rho_u=1$ and $\rho K=\rho$. Thus $K$ is the transition matrix of an irreducible positive recurrent Markov chain with stationary distribution $\rho$. Since $K_{uv}>0$ if and only if $A_{uv}>0$, this chain has the same period as $\Sigma$. By \eqref{eq:transition-stationary} and induction,
\begin{equation}\label{eq:matrix-identity}
    (K^n)_{uv}=\lambda^{-n}\frac{\ell_v}{\ell_u}(A^n)_{uv},\qquad n\geq0.
\end{equation}

The measure $\mu$ is the stationary Markov measure determined by $K$ and $\rho$; see \cite{Gurevich70}. In particular, for every admissible cylinder $[u_0,\ldots,u_n]_m$, we have
\begin{equation}\label{eq:MME}
    \mu([u_0,\ldots,u_n]_m)=\rho_{u_0}\prod_{j=m}^{n+m-1}K_{u_{j-m}u_{j-m+1}}=\lambda^{-n}r_{u_0}\ell_{u_n},
\end{equation}
where an empty product equals one. In particular, the measure of an admissible cylinder depends only on the length of the path and the starting and ending vertices. 

We will need the following convergence theorem for Markov chains. It states that, starting from any fixed state, the distribution at time $n$ converges to the stationary distribution.

\begin{theorem}[\cite{Durrett99ESP}]\label{markov-aperiodic}
Let $\Sigma$ be a positive recurrent aperiodic irreducible countable Markov shift with finite Perron value. Let $K$ and $\rho$ be its associated stochastic matrix and stationary probability distribution, respectively. Then, for every $u,v\in\mathcal V$,
\[
    \lim_{n\to\infty}(K^n)_{uv}=\rho_v.
\]
\end{theorem}

\subsection{MME in the symbolic model}

Next, we prove the desired properties of MME for a countable Markov shift with an upper bound on the number of periodic points. This result provides the symbolic model for the proof of our main theorem.

\begin{proposition}\label{prop:symbolic}
    Let $(\Sigma,\sigma)$ be a countable Markov shift and set $\lambda>1$. Assume that $\Fix(\sigma^n)$ is finite for every $n\geq 1$ and
    \begin{equation}\label{eq:symbolic-upper}
        \limsup_{n\to\infty} \lambda^{-n}\#\Fix(\sigma^n)\leq 1.
    \end{equation}
    Then there is at most one $\sigma$-invariant Borel probability measure $\nu$ with $h_\nu(\sigma)=\log\lambda$.

    If such a measure exists, it is ergodic, the unique measure of maximal entropy and is supported on an aperiodic positive recurrent irreducible Markov subshift $\Sigma(\cC)$. Moreover, $(\Sigma,\nu,\sigma)$ is Bernoulli, and the measures
    \[
        \eta_n:=\lambda^{-n}\sum_{x\in\Fix(\sigma^n)}\delta_x
    \]
    converge weakly to $\nu$ as $n\to\infty$. 
\end{proposition}

We first prove a preliminary version of the uniqueness for ergodic invariant measures of entropy $\log\lambda$.

\begin{lemma}\label{lemma:unique-ergodic}
    Under the assumptions of \cref{prop:symbolic}, there is at most one ergodic invariant probability measure of entropy $\log\lambda$.
\end{lemma}

\begin{proof}
Suppose $\nu$ is an ergodic invariant probability measure. By \cref{lem:ergodic-component}, $\nu$ is supported on an irreducible subshift $\Sigma(\cC)$. Let $A_\cC$ be its adjacency matrix. For every vertex $u$ of $\cC$,
\[
    (A_\cC^n)_{uu}=\#\{x\in\Fix(\sigma^n|_{\Sigma(\cC)}):x_0=u\}\leq\#\Fix(\sigma^n).
\]
By \eqref{eq:symbolic-upper}, the Perron value of $\Sigma(\cC)$ is at most $\lambda$. The variational principle (\cref{variational}) therefore gives
\begin{equation}\label{eq:up-boun}
    h_\nu(\sigma)\leq h(\Sigma(\cC))\leq\log\lambda.
\end{equation}

Now suppose $h_\nu(\sigma)=\log\lambda$. By \eqref{eq:up-boun} and \cref{thm:MME-posi recurr}, $\Sigma(\cC)$ is positive recurrent with Perron value $\lambda$, and $\nu$ is its unique measure of maximal entropy. Let $q$ be the period of $\Sigma(\cC)$. By \cref{lemma:lower bound},
\[
    q\leq\liminf_{\substack{n\to\infty\\q\mid n}}\lambda^{-n}\#\Fix(\sigma^n|_{\Sigma(\cC)}) \leq\limsup_{n\to\infty} \lambda^{-n}\#\Fix(\sigma^n)\leq1.
\]
Hence $q=1$, and we obtain
\begin{equation}\label{eq:preiodic}
 \lambda^{-n}\#\Fix(\sigma^n|_{\Sigma(\cC)})\longrightarrow 1, \qquad \lambda^{-n}\#\Fix(\sigma^n)\longrightarrow 1.
\end{equation}

Suppose that there are two distinct ergodic invariant probability measures $\nu_1,\nu_2$ of entropy $\log\lambda$. By \cref{lem:ergodic-component} and \cref{thm:MME-posi recurr}, they are carried by distinct irreducible subshifts $\Sigma(\cC_1)$ and $\Sigma(\cC_2)$. Applying \eqref{eq:preiodic} to both gives
\[
    \liminf_{n\to\infty}\lambda^{-n}\#\Fix(\sigma^n) \geq \lim_{n\to\infty}\lambda^{-n} \sum_{i=1}^{2}\#\Fix(\sigma^n|_{\Sigma(\cC_i)})=2,
\]
contradicting \eqref{eq:symbolic-upper}. 
\end{proof}

\begin{proof}[Proof of \cref{prop:symbolic}]

Let $\mu$ be any invariant probability measure and write its ergodic decomposition as $\mu=\int\mu_e\,d\tau(e)$. The entropy formula and \eqref{eq:up-boun} give
\[
    h_\mu(\sigma) =\int h_{\mu_e}(\sigma)\,d\tau(e)\leq\log\lambda.
\]
If $h_\mu(\sigma)=\log\lambda$, then $h_{\mu_e}(\sigma)=\log\lambda$ for $\tau$-almost every $e$. By \cref{lemma:unique-ergodic}, the measures $\mu_e$ coincide for $\tau$-almost every $e$. We write $\mu_e=\nu$ for $\tau$-almost every $e$. Thus $\mu=\nu$ is ergodic and is the unique invariant probability measure of maximal entropy $\log\lambda$. 

The measure $\nu$ is supported on an irreducible subshift $\Sigma(\cC)$ which is aperiodic and therefore topologically mixing, positive recurrent, and satisfies \eqref{eq:preiodic}. The Bernoulli property follows from \cite[Theorem~3.1]{Sarig11JMD}, applied to the zero potential on $\Sigma(\cC)$. Since $\supp(\nu)\subset \Sigma(\cC)$, the system $(\Sigma,\sigma,\nu)$ is also Bernoulli.

It remains to prove equidistribution. By \cref{markov-aperiodic} and \eqref{eq:matrix-identity}, we have
\begin{equation}\label{eq:matrix-limit}
    \lambda^{-n}(A^n)_{uv}\longrightarrow r_{v}\ell_{u}.
\end{equation} Fix an admissible cylinder $C=[u_0,\cdots, u_m]_0$. For $n>m$, completing the path $(u_0,\ldots,u_m)$ to a fixed point of $\sigma^n$ amounts to choosing a path of length $n-m$ from $u_m$ to $u_0$. Consequently, we have
\[
    \#\bigl(\Fix(\sigma^n)\cap C\bigr)=(A^{n-m})_{u_m,u_0}.
\]
Using \eqref{eq:matrix-limit} and \eqref{eq:MME}, we obtain
\begin{align*}
 \lim_{n\to\infty}\eta_n(C)
    &=\lim_{n\to\infty}\lambda^{-m}\lambda^{-(n-m)}(A^{n-m})_{u_m,u_0}\\
    &=\lambda^{-m}r_{u_0}\ell_{u_m}
    =\nu(C).
\end{align*}
Since both $\eta_n$ and $\nu$ are shift-invariant, we can obtain the convergence for cylinders on any finite interval. Note that the intersection of finite cylinders is empty or a disjoint union of several cylinders, the above convergence result also holds for finite union of cylinders because of inclusion-exclusion principle. Because the cylinders on finite intervals form a countable topology basis for $\Sigma$, for any open set $G\subset\Sigma$, we can write $G=\bigcup_{j\geq1}C_j$, where each $C_j$ is a cylinder contained in $G$. For every $N\geq1$,
\[
    \liminf_{n\to\infty}\eta_n(G)\geq\lim_{n\to\infty}\eta_n\Bigl(\bigcup_{j=1}^{N}C_j\Bigr)=\nu\Bigl(\bigcup_{j=1}^{N}C_j\Bigr).
\]
Letting $N\to\infty$ yields $\liminf_{n\to\infty}\eta_n(G)\geq\nu(G)$ for any open subset $G$. This implies $\eta_n$ converges to $\nu$. Therefore, we finish the proof.
\end{proof}

\begin{remark}\label{rmk:xi}
    It follows from \eqref{eq:preiodic} that $\#\Fix(\sigma^n)\setminus\Fix(\sigma^n|_{\Sigma(\cC)})=o(\lambda^n)$. Thus if we define $\xi_n:=\lambda^{-n}\sum_{x\in\Fix(\sigma^n|_{\Sigma(\cC)})}\delta_x$, then they also converge to $\nu$ as $n\to\infty$.
\end{remark}

\section{Natural extensions and coding}\label{sec:symbolic}

\subsection{Natural extensions and Lyapunov exponents}\label{sec:natural-extension}

Let $f$ be a dominant meromorphic map on $X$ and let $\cS=I(f) \cup C(f)$ be as in \cref{sec:intro}. Put $U:=X\setminus\cS$. We give a concise introduction of the natural extension of $f$, following \cite[\S2.1]{DDG3ENS} and \cite[\S2.2]{LOP26JEMS}. Let $d$ be the distance induced by a fixed K\"ahler metric on $X$, and define the \emph{regular history space}
\[
    \widehat X_{\mathrm{reg}} :=\left\{(x_j)_{j\in\mathbb Z}\in U^{\mathbb Z}: f(x_j)=x_{j+1}\text{ for every }j\in\mathbb Z\right\}.
\]
This is a closed subset of the Polish space $U^{\mathbb Z}$ and hence a standard Borel space. The left shift and the zeroth-coordinate projection are
\[
    \widehat f((x_j)_j):=(x_{j+1})_j, \qquad \vartheta_0((x_j)_j):=x_0.
\]
They satisfy $\vartheta_0\circ\widehat f=f\circ\vartheta_0$. The metric $\widehat d(\widehat x,\widehat y):=\sup_{j\leq0}\{2^j d(x_j,y_j)\}$ induces the product topology on $\widehat X_{\mathrm{reg}}$ for which $\widehat f$ is a homeomorphism and $\vartheta_0$ is $1$-Lipschitz.

For every invariant probability measure $\nu$ with $\nu(\cS)=0$, there is a unique $\widehat f$-invariant probability measure $\widehat\nu$ on $\widehat X_{\mathrm{reg}}$ such that $(\vartheta_0)_*\widehat\nu=\nu$. We call $\widehat\nu$ the \emph{natural-extension lift} of $\nu$. The measure $\widehat\nu$ is ergodic if and only if $\nu$ is ergodic, and their Kolmogorov--Sinai entropies agree:
\[
    h_{\widehat\nu}(\widehat f)=h_\nu(f).
\]
These are standard properties of the natural extension; see \cite{Rohl}.

Let $N:=\vartheta_0^*TX$ be the pulling back of the tangent vector bundle $TX$ to $ \widehat{X}_{\rm reg}$, with fibre $N_{\hat x}=T_{x_0}X$ at $\hat x=(x_j)_{j\in\Z}$. Since every coordinate of $\hat x$ lies outside $\cS$, the derivative defines an invertible complex-linear cocycle
\[
    \widehat{Df}_{\hat x}:=Df_{x_0}\colon N_{\hat x}\longrightarrow N_{\widehat f(\hat x)}.
\]
Set $\widehat{Df}^{(0)}_{\hat x}:=\mathrm{id}_{N_{\hat x}}$. For $n\geq1$, its forward and backward iterates are
\[
    \widehat{Df}^{(n)}_{\hat x}=D(f^n)_{x_0},\qquad\widehat{Df}^{(-n)}_{\hat x}=\bigl(D(f^n)_{x_{-n}}\bigr)^{-1}.
\]

The derivative estimate \eqref{eq:derivative-growth} in \hyperref[app:estimates]{Appendix} ensures that the cocycle $\widehat{Df}$ satisfies the integrability assumptions of Oseledets' theorem for every adapted invariant probability measure $\nu$. Indeed, write $\log^+t:=\max\{0,\log t\}$. Since the $x_0$- and $x_{-1}$-marginals of $\widehat\nu$ are both $\nu$, we obtain
\[
\begin{aligned}
    &\int_{\widehat X_{\mathrm{reg}}}\left(\log^+\|\widehat{Df}^{(1)}_{\hat x}\|+\log^+\|\widehat{Df}^{(-1)}_{\hat x}\|\right)\,\mathrm d\widehat\nu(\hat x)\\
    &\qquad=\int_U\left(\log^+\|Df_x\|+\log^+\|(Df_x)^{-1}\|\right)\,\mathrm d\nu(x)\\
    &\qquad\leq 2\log C_0+2b\int_X-\log\rho(x)\,\mathrm d\nu(x)<\infty.
\end{aligned}
\]

Suppose now that $\nu$ is ergodic. By the non-invertible version of Oseledets' theorem, there are finite real numbers $\gamma_1>\cdots>\gamma_r$, and on an $\widehat f$-invariant set of full $\widehat\nu$-measure, a measurable splitting
\[
    N_{\hat x}=F_1(\hat x)\oplus\cdots\oplus F_r(\hat x)
\]
such that
\[
    \widehat{Df}_{\hat x}F_i(\hat x)=F_i(\widehat f(\hat x)),\qquad \lim_{n\to\pm\infty}\frac1n\log\|\widehat{Df}^{(n)}_{\hat x}v\|=\gamma_i
\]
for every $0\ne v\in F_i(\hat x)$. Ergodicity implies that the dimensions of these subspaces are constant $\widehat\nu$-almost everywhere.

Each $F_i(\hat x)$ is a complex subspace. Repeating each $\gamma_i$ according to $\dim_{\mathbb C}F_i$, we obtain the \textit{Lyapunov exponents} $\chi_1(\nu)\geq\cdots\geq\chi_k(\nu)$, counted with complex multiplicity. These coincide with the Lyapunov exponents of $Df$ with respect to $\nu$.

For $\chi>0$, we call $\nu$ \emph{$\chi$-hyperbolic} if $\min_i|\chi_i(\nu)|>\chi$. In this case there are no zero exponents, and the Oseledets stable and unstable subspaces
\begin{equation}\label{eq:oseledet-un/stable}
    E_{\mathrm{Os}}^s(\hat x):=\bigoplus_{\gamma_i<0}F_i(\hat x),\qquad E_{\mathrm{Os}}^u(\hat x):=\bigoplus_{\gamma_i>0}F_i(\hat x)
\end{equation}
give a measurable invariant splitting $N_{\hat x}=E_{\mathrm{Os}}^s(\hat x)\oplus E_{\mathrm{Os}}^u(\hat x)$. Vectors in $E_{\mathrm{Os}}^s(\hat x)$ contract exponentially under forward iteration, while vectors in $E_{\mathrm{Os}}^u(\hat x)$ contract exponentially under backward iteration.

\subsection{Coding of the natural extension}

Symbolic coding for nonuniformly hyperbolic systems has developed substantially in recent years. Relevant advances include Ben Ovadia's construction for diffeomorphisms in arbitrary dimension \cite{Ovadia18JMD}, the coding for surface diffeomorphisms by Buzzi--Crovisier--Sarig \cite{BCS22Annals}, and the extension to noninvertible maps with singularities in arbitrary dimension by Ara\'ujo--Lima--Poletti \cite{ALP24MAMS}. We use the following consequence of \cite[Theorem~3.1]{LOP26JEMS}. As in \cite{ALP24MAMS}, its application requires geometric and regularity assumptions \textup{(A1)--(A7)}, stated in \cite[\S3.1]{LOP26JEMS}. These are verified in our setting in \hyperref[app:estimates]{Appendix}.

\begin{theorem}[\cite{LOP26JEMS}]\label{thm:mme-coding}
    Let $(X,f)$ be as above. For every $\chi>0$, there are a locally compact countable Markov shift $(\Sigma_1,\sigma_1)$ and a H\"older continuous map $\pi: \Sigma_1^\#\longrightarrow\widehat X_{\mathrm{reg}}$ such that $\pi\circ\sigma_1=\widehat f\circ\pi$ with the following properties.
\begin{enumerate}[label=\textup{(\arabic*)}]
    \item The map $\pi$ is finite-to-one.

    \item For every adapted invariant $\chi$-hyperbolic probability measure $\nu$ on $X$, let $\widehat{\nu}$ be its natural-extension lift on $\widehat{X}_{\rm reg}$. Then we have $\widehat\nu\bigl(\pi(\Sigma_1^\#)\bigr)=1$. 

    \item Bowen Property: there is a reflexive symmetric relation $\sim$ on the alphabet $V_1$ such that $\{w\in V_1:v\sim w\}$ is finite for every $v\in V_1$, and, for all $z,w\in\Sigma_1^\#$,
    \[
        \pi(z)=\pi(w) \quad\Longleftrightarrow\quad z_j\sim w_j \text{ for every }j\in\mathbb Z.
    \]

    \item Over $\pi(\Sigma_1^\#)$ there is a
    $\widehat{Df}$-invariant splitting into complex subspaces
    \[
        N_{\hat x}=T_{\vartheta_0(\hat x)}X=E^s(\hat x)\oplus E^u(\hat x),
    \]
    such that
    \begin{equation}\label{eq:coding-local-contraction}
            \limsup_{j\to\infty}\frac1j \log\bigl\|\widehat{Df}^{(j)}_{\hat x}|_{E^s(\hat x)}\bigr\|\leq-\frac{\chi}{2},\quad \limsup_{j\to\infty}\frac1j\log\bigl\|\widehat{Df}^{(-j)}_{\hat x}|_{E^u(\hat x)}\bigr\|\leq-\frac{\chi}{2}.
    \end{equation}
    The maps $z\mapsto E^{s/u}(\pi(z))$ are H\"older continuous.
\end{enumerate}
\end{theorem}

\section{Proof of the main results}\label{sec:proof-main}

We start the proof of \cref{thm:main}. Denote by $\lambda:=d_p(f)>1$. Log-concavity implies that the dynamical degrees strictly increase up to $p$ and strictly decrease afterwards. By \cite[Corollary~3]{deThelin08Invent}, every adapted ergodic invariant measure $\nu$ with $h_\nu(f)=\log\lambda$ satisfies
\begin{equation}\label{eq:dethelin-lyapunov}
    \chi_p(\nu)\geq\frac12\log\frac{\lambda}{d_{p-1}(f)}>0,\qquad \chi_{p+1}(\nu)\leq-\frac12\log\frac{\lambda}{d_{p+1}(f)}<0.
\end{equation}
Thus every such measure has exactly $p$ positive and $k-p$ negative complex Lyapunov exponents. Fix
\[
    0<\chi_0<\frac12\min\left\{\log\frac{\lambda}{d_{p-1}(f)},\log\frac{\lambda}{d_{p+1}(f)}\right\}.
\]
All these measures are $\chi_0$-hyperbolic.

Apply \cref{thm:mme-coding} with $\chi=\chi_0$. Buzzi's Main Theorem \cite{Buzzi20JMD}, applied as in \cite[Theorem~9.1]{ALP24MAMS}, gives a countable Markov shift $(\Sigma,\sigma)$ and a $1$-Lipschitz map $\phi\colon\Sigma\to\Sigma_1^\#$ such that $\widetilde\Pi:=\pi\circ\phi$ is an injective semiconjugacy whose image carries the natural-extension lift of every adapted ergodic invariant measure of entropy $\log\lambda$. Set
\[
    \Pi:=\vartheta_0\circ\widetilde\Pi.
\]
Then $\Pi$ is H\"older continuous, $\Pi\circ\sigma=f\circ\Pi$, and $\widetilde\Pi(z)=\bigl(\Pi(\sigma^jz)\bigr)_{j\in\Z}$. We have the following commutative diagram:
\[
    \begin{tikzcd}[row sep=3em, column sep=4em]
    \Sigma \arrow[d, "\phi"'] \arrow[dr, "\widetilde\Pi"'] \arrow[drr, bend left=15, "\Pi"] & & \\
    \Sigma_1^\# \arrow[r, "\pi"'] & \widehat X_{\mathrm{reg}} \arrow[r, "\vartheta_0"'] & X
    \end{tikzcd}
\]

\begin{lemma}\label{lemma-1}
    Every adapted ergodic invariant probability measure $\nu$ with $h_\nu(f)=\log\lambda$ has a unique $\sigma$-invariant lift $m$ satisfying $\Pi_*m=\nu$. The map $\widetilde\Pi$ identifies $(\Sigma,m,\sigma)$ with $(\widehat{X}_{\rm reg}, \widehat{\nu}, \widehat{f})$ modulo null sets. In particular, $m$ is ergodic and $h_m(\sigma)=h_\nu(f)$.
\end{lemma}

\begin{proof}
Let $\widehat\nu$ be the natural-extension lift of $\nu$. Since $\widetilde\Pi$ is an injective Borel map between standard Borel spaces, it is a Borel conjugacy onto its image and $\widehat{\nu}(\widetilde\Pi(\Sigma))=1$. Hence the pullback $m:=\widetilde\Pi^*\widehat\nu$ is a $\sigma$-invariant probability measure satisfying $\Pi_*m=\nu$, and $\widetilde\Pi$ gives the asserted isomorphism with the natural extension of $(X,\nu,f)$. Since natural extension preserves entropy and ergodicity, $m$ is ergodic and
\[
    h_m(\sigma)=h_{\widehat\nu}(\widehat f)=h_\nu(f).
\]
If $m'$ is another $\sigma$-invariant lift of $\nu$, then $\widetilde\Pi_*m'$ is $\widehat f$-invariant and $(\vartheta_0)_*(\widetilde\Pi_*m')=\nu$. Uniqueness of the natural-extension lift gives $\widetilde\Pi_*m'=\widehat\nu$, and injectivity of $\widetilde\Pi$ therefore yields $m'=m$.
\end{proof}

To establish the periodic-point upper bound required by \cref{prop:symbolic}, we first examine the images of periodic points of $\sigma$. In \cref{thm:mme-coding}-(4), for $z\in \Sigma$, set $E_z^{s/u}:=E^{s/u}(\widetilde\Pi(z))$. Then we have
\[
    T_{\Pi(z)}X=E_z^s\oplus E_z^u.
\]
These subspaces depend continuously on $z$, since the maps $\Sigma_1^{\#}\ni z_1\mapsto E^{s/u}(\pi(z_1))$ are H\"older continuous and $\phi$ is Lipschitz. They inherit the invariance and contraction estimates \eqref{eq:coding-local-contraction}, with $\chi=\chi_0$.

\begin{lemma}\label{lem:periodic-projection}
    For every $n\geq1$, the map $\Pi$ is injective on $\Fix(\sigma^n)$, and each $z\in\Fix(\sigma^n)$ projects to a saddle $n$-periodic point $x:=\Pi(z)$ of $f$, and $x$ is isolated and of multiplicity one. Its stable and unstable spaces are $E_z^s$ and $E_z^u$, respectively. The eigenvalues of $(Df^n)_x|_{E_z^s}$ have modulus at most $e^{-\chi_0n/2}$, while those of $(Df^n)_x|_{E_z^u}$ have modulus at least $e^{\chi_0n/2}$.
\end{lemma}

\begin{proof}
For every $n\geq 1$, $\vartheta_0$ is injective on $\Fix(\widehat{f}^n)$. Since $\widetilde{\Pi}$ is injective, it follows that $\Pi$ is injective on $\Fix(\sigma^n)$. 

Fix $n\geq 1$ and $z\in\Fix(\sigma^n)$. Write $\hat x:=\widetilde\Pi(z)$. Then $f^n(x)=x$ and its orbit avoids $\cS$. Thus $f^n$ is holomorphic near $x$, and $L:=(Df^n)_x$ is invertible and preserves $E_z^s$ and $E_z^u$. Denote by $\mathrm{spr}(\cdot)$ the spectral radius. Since $(Df^n)_x=(\widehat{Df}^{(n)})_{\hat x}$, by \eqref{eq:coding-local-contraction} and Gelfand's formula, we have
\[
    \mathrm{spr}(L|_{E_z^s})\leq e^{-\chi_0n/2},\qquad \mathrm{spr}\bigl((L|_{E_z^u})^{-1}\bigr)\leq e^{-\chi_0n/2},
\]
whenever the corresponding subspace is nonzero. Hence $L$ is hyperbolic, its stable and unstable subspaces are $E_z^s$ and $E_z^u$, and its eigenvalues satisfy the asserted bounds. 
\end{proof}

By \cref{lem:periodic-projection} and \eqref{eq:criterion-upper}, we have
\[
    \#\Fix(\sigma^n)\leq\#P_n(f)\leq\lambda^n+o(\lambda^n).
\]
We can therefore apply \cref{prop:symbolic}.

We first prove assertion \textup{(1)}. By \cref{lemma-1}, every adapted ergodic invariant measure $\nu$ of entropy $\log\lambda$ has a $\sigma$-invariant lift $m$ satisfying
\[
    \Pi_*m=\nu,\qquad h_m(\sigma)=\log\lambda.
\]
By \cref{prop:symbolic}, there is at most one such $m$, and hence at most one such ergodic $\nu$. If $\mu$ is an adapted invariant probability measure with $h_\mu(f)=\log\lambda$. Consider its ergodic decomposition $\mu=\int\nu_e\d\tau(e)$. It is easy to see that for $\tau$-almost every $e$, $\nu_e$ is adapted with $h_{\nu_e}(f)=\log\lambda$. Consequently, $\nu_e=\nu$ for $\tau$-almost every $e$. Therefore $\mu=\nu$, and it is also ergodic.

Suppose now that this measure $\mu$ exists. By \eqref{eq:dethelin-lyapunov}, $\mu$ is hyperbolic with exactly $p$ positive and $k-p$ negative Lyapunov exponents, counted with complex multiplicity. Denote its lift as in \cref{lemma-1} by $m$. By \cref{prop:symbolic}, $m$ is the measure of maximal entropy of an aperiodic positive recurrent component $\Sigma(\cC)$, $\supp(m)=\Sigma(\cC)$, and $(\Sigma(\cC),m,\sigma)$ is Bernoulli. By \cref{lemma-1}, the natural extension of $(X,\mu,f)$ is therefore Bernoulli.

If $f$ is bimeromorphic, its restriction to $U$ is injective. Indeed, two distinct regular preimages of the same point would yield two disjoint local inverse branches over a common open set, contradicting generic injectivity. Hence $\vartheta_0$ is injective on $\widehat X_{\mathrm{reg}}$ and identifies $(X,\mu,f)$ with its natural extension modulo null sets. Thus $(X,\mu,f)$ is also Bernoulli, proving assertion \textup{(2)}.

To prove assertion \textup{(3)}, we first determine the dimensions of the invariant subspaces throughout $\Sigma(\cC)$.

\begin{lemma}\label{lem:stable-index}
    For every $z\in\Sigma(\cC)$, the subspaces $E_z^s$ and $E_z^u$ are of complex dimensions $k-p$ and $p$, respectively.
\end{lemma}

\begin{proof}
    By the discussion in \cref{sec:natural-extension}, for $\widehat{\mu}$-almost every $\hat x$, we have the decomposition $T_{\vartheta_0(\hat x)}X=E_{\mathrm{Os}}^s(\hat x)\oplus E_{\mathrm{Os}}^u(\hat x)$. Since $\widetilde\Pi_*m=\widehat\mu$, this holds for $m$-almost every $z$ with $\hat x=\widetilde\Pi(z)$. Comparing with the decomposition in \cref{thm:mme-coding}-(4), it is easy to see that 
    \[
        E_z^s=E_{\mathrm{Os}}^s(\widetilde\Pi(z)),\qquad E_z^u=E_{\mathrm{Os}}^u(\widetilde\Pi(z)).
    \]
    In particular, the asserted conclusion holds $m$-almost everywhere. Finally, since $\supp(m)=\Sigma(\cC)$, the H\"older continuity of the subspaces extends the conclusion to every point of $\Sigma(\cC)$.
\end{proof}

\begin{proof}[End of the proof of \cref{thm:main}]
Set $\chi:=\chi_0/4$. By \cref{lem:stable-index,lem:periodic-projection}, every $z\in\Fix(\sigma^n|_{\Sigma(\cC)})$ projects to a  saddle periodic point with $p$ unstable and $k-p$ stable multipliers satisfying the stated bounds. Hence $\Pi(\Fix(\sigma^n|_{\Sigma(\cC)}))\subset P_{n,\chi}(f)$. By \cref{prop:symbolic,rmk:xi}, the component periodic measures satisfy
\begin{equation}\label{eq:component-periodic-measures}
    \xi_n:=\lambda^{-n}\sum_{z\in\Fix(\sigma^n|_{\Sigma(\cC)})}\delta_z\longrightarrow m,\qquad \|\xi_n\|\longrightarrow1.
\end{equation}
Continuity of $\Pi$ gives $\zeta_n:=\Pi_*\xi_n\longrightarrow\mu$ and $\|\zeta_n\|\to1$. By \cref{lem:periodic-projection}, the points counted by $\zeta_n$ are distinct and have multiplicity one. For any sequence $P_{n,\chi}(f)\subset Q_n\subset P_n(f)$, define $\eta_n:=\lambda^{-n}\sum_{x\in Q_n}\delta_x$, where points are counted with multiplicity. Since all points counted by $\zeta_n$ belong to $P_{n,\chi}(f)$, we have $0\leq\zeta_n\leq\eta_n$. Therefore, \eqref{eq:criterion-upper} gives
\[
    0\leq\|\eta_n-\zeta_n\|\leq\lambda^{-n}\#P_n(f)-\|\zeta_n\|\longrightarrow0.
\]
It follows that $\eta_n\to\mu$, proving \eqref{eq:equidist}.
\end{proof}

The recent work of Hu \cite{Hu26} on the growth of dynamical degrees of a meromorphic self-map, combined with the machinery we just used, allows us to prove a finiteness result of the MMEs. 

\begin{proof}[Proof of \cref{prop:finite MMEs}]

Let $\Gamma_n$ be the graph of $f^n$ in $X\times X$. By \cite[Corollary~5.1]{Hu26}, we have $\|[\Gamma_n]\|=O(\lambda^n)$. The intersection estimate in \cite[Proposition 3.1]{Vu21IMRN}, applied to $\Gamma_n$ and the diagonal, therefore gives $\#P_n(f)\leq C\lambda^n$ for some $C>0$ and every $n\geq1$.

Let $\nu_1,\ldots,\nu_r$ be any finite collection of distinct adapted ergodic invariant probability measures of entropy $\log\lambda$. The arguments in the proof of \cref{thm:main} give a common countable Markov shift $(\Sigma,\sigma)$ satisfying
\begin{equation}\label{eq:finite-period}
    \#\Fix(\sigma^n)\leq\#P_n(f)\leq C\lambda^n,
\end{equation}
together with ergodic lifts $m_i$ satisfying $\Pi_*m_i=\nu_i$ and $h_{m_i}(\sigma)=\log\lambda$ for every $1\leq i\leq r$. Since $m_i$ is ergodic, \cref{lem:ergodic-component} gives an irreducible Markov subshift $\Sigma(\cC_i)$ such that $m_i(\Sigma(\cC_i))=1$. As in the proof of \cref{lemma:unique-ergodic}, the upper bound \eqref{eq:finite-period} implies that its Perron value is $\lambda$ and that $m_i$ is the unique measure of maximal entropy on $\Sigma(\cC_i)$ for every $i$. We deduce that the $\Sigma(\cC_i)$'s are disjoint from each other.

Let $q_i$ be the period of $\Sigma(\cC_i)$, and let $q$ be a common multiple of these periods. By \cref{lemma:lower bound} and the disjointness,
\[
    r\leq\sum_{i=1}^r q_i\leq\liminf_{\substack{n\to\infty\\q\mid n}}\lambda^{-n}\#\Fix(\sigma^n)\leq C.
\]
Since $C$ is independent of the chosen finite collection, there are only finitely many such measures.
\end{proof}

\section{Applications}\label{sec:applications}

We give a more detailed description about each of the dynamical systems mentioned in the introduction. We then prove \cref{prop:applications}.

\subsection{Maps of surfaces with small topological degree}
\label{sec:applications-surfaces}

Let $f\colon X\dashrightarrow X$ be a dominant meromorphic map of a smooth connected projective surface with $\lambda:=d_1(f)>d_2(f)$. We consider the setting of Diller--Dujardin--Guedj \cite{DDG1Indiana,DDG2CMH,DDG3ENS}; see also \cite{DF01AJM}. Assume that
\begin{enumerate}
    \item[\textup{(H1)}] The map $f$ is algebraically stable, i.e., $(f^n)^*=(f^*)^n$ on $H^{1,1}(X)$ for every $n\geq1$.
    \item[\textup{(H2)}] The map $f$ has finite dynamical energy in the sense of \cite{DDG2CMH}.
\end{enumerate}

Let $\omega_X$ be a Kähler form on $X$. Under \textup{(H1)}, the limits
\[
    T^+:=\lim_{n\to\infty}\lambda^{-n}(f^n)^*\omega_X,\qquad T^-:=\lim_{n\to\infty}\lambda^{-n}(f^n)_*\omega_X
\]
exist in the sense of currents. They are nonzero positive closed $(1,1)$-currents satisfying
\[
    f^*T^+=\lambda T^+,\qquad f_*T^-=\lambda T^-.
\]
We rescale $\omega_X$ so that the cohomological intersection $[T^+]\cdot[T^-]$ equals $1$. Assumption \textup{(H2)} ensures that the wedge product
\[
    \mu:=T^+\wedge T^-
\]
is well-defined and is a probability measure. This measure is invariant, mixing, and satisfies
\[
    h_\mu(f)=h_{\mathrm{top}}(f)=\log\lambda.
\]
Its natural extension has local product structure and is Bernoulli.

We further assume that
\begin{enumerate}
    \item[\textup{(H3)}] The measure $\mu$ is adapted, i.e., $\log\dist(\cdot,\cS)\in L^1(\mu)$.
\end{enumerate}

\subsection{Dominant rational maps on projective space}
\label{sec:applications-projective}

Let $f\colon\Pb^k\dashrightarrow\Pb^k$ be a dominant rational map. Normalize the Fubini--Study form $\omega$ by $\int_{\Pb^k}\omega^k=1$, and put
\[
    \delta_j(f) :=\int_{\Pb^k}f^*(\omega^j)\wedge\omega^{k-j}, \qquad d_j(f):=\lim_{n\to\infty}\delta_j(f^n)^{1/n}.
\]
Suppose that, for some $1\leq p\leq k-1$, the map $f$ is $p$-algebraically stable and $\lambda:=d_p(f)>\max_{j\ne p}d_j(f)$. Since $H^{p,p}(\Pb^k)$ is one-dimensional, we have
\begin{equation}\label{eq:projective-stability}
    \delta_p(f^n)=\delta_p(f)^n=\lambda^n.
\end{equation}

An important subclass consists of the birational maps studied by de Th\'elin--Vigny in \cite{dTH-Vigny}. Let
\[
    d:=\deg f\geq2, \qquad e:=\deg f^{-1}, \qquad I^+:=I(f), \qquad I^-:=I(f^{-1}).
\]
Assume that, for some $1\leq s\leq k-1$,
\begin{equation}\label{eq:dtv-dimensions}
    \dim I^+=k-s-1, \qquad \dim I^-=s-1,
\end{equation}
and that
\begin{equation}\label{eq:dtv-separation}
    \left(\bigcup_{n\geq0}f^{-n}(I^+)\right) \cap \left(\bigcup_{n\geq0}f^n(I^-)\right) =\varnothing.
\end{equation}
We also impose their finite-energy condition. Choose negative quasi-potentials $u^\pm$ satisfying
\[
    d^{-1}f^*\omega=\omega+dd^cu^+, \qquad e^{-1}(f^{-1})^*\omega=\omega+dd^cu^-,
\]
and assume
\begin{equation}\label{eq:dtv-energy}
    \begin{aligned}
        &\sum_{n\geq0}d^{-sn}\int_{[f^n(I^-)]}u^+\,f^*(\omega^{s-1})>-\infty,\\
        &\sum_{n\geq0}e^{-(k-s)n}\int_{[f^{-n}(I^+)]}u^-\,(f^{-1})^*(\omega^{k-s-1})>-\infty.
    \end{aligned}
\end{equation}
The integration on the cycles are computed with multiplicity. This is the condition in \cite[Definition~3.1.9]{dTH-Vigny}, up to fixed positive normalization factors; see also \cite[equation~(1)]{DV24}.

Under these assumptions, $f$ is $j$-algebraically stable for every $j$, and 
\begin{equation}\label{eq:dtv-degrees}
    d_j(f)=
    \begin{cases}
        d^j,&0\leq j\leq s,\\
        e^{k-j},&s\leq j\leq k,
    \end{cases}
\end{equation}
where $d^s=e^{k-s}$; see \cite[Proposition~3.1.5 and Lemma~3.1.6]{dTH-Vigny}. Thus $\lambda:=d_s(f)=d^s$ is strictly dominant. Under these hypotheses, the limits
\[
    T_s^+:=\lim_{n\to\infty}\lambda^{-n}(f^n)^*(\omega^s),\qquad T_{k-s}^-:=\lim_{n\to\infty}\lambda^{-n}(f^n)_*(\omega^{k-s})
\]
exist in the sense of currents. They are positive closed currents of bidegrees $(s,s)$ and $(k-s,k-s)$, respectively, both of mass one, satisfying
\[
    f^*T_s^+=\lambda T_s^+,\qquad f_*T_{k-s}^-=\lambda T_{k-s}^-.
\]
Their wedge product
\begin{equation}\label{eq:dtv-measure}
    \mu:=T_s^+\wedge T_{k-s}^-,
\end{equation}
is well-defined. It is an invariant mixing probability measure with entropy $\log\lambda$; see \cite[Theorems~3.4.1, 3.4.15 and~3.4.19]{dTH-Vigny}. Its adaptedness follows from the argument preceding \cite[Theorem 3.4.21]{dTH-Vigny}.

\begin{proof}[Proof of \cref{prop:applications}]

For the maps considered in \cref{sec:applications-surfaces}, we have $p=1$ and $d_1(f)>d_2(f)$, so it follows from \cite[Theorem 1.2]{DNT17BLMS}.

For \cref{sec:applications-projective}, let $\Gamma_n$ be the closure of the graph of $f^n$, and let $\Delta$ be the diagonal in $\Pb^k\times\Pb^k$. For $A\in\operatorname{PGL}(k+1,\mathbb C)$, set
\[
    \Delta_A:=\{(x,Ax):x\in\Pb^k\}.
\]
For general $A$ arbitrarily close to the identity, $\Gamma_n\cap\Delta_A$ is a proper intersection. Indeed, for each $(x,y)\in\Gamma_n$, $\{ A\in \operatorname{PGL}(k+1,\mathbb C) | Ax=y \}$ is a subvariety of codimension $k$ in $\operatorname{PGL}(k+1,\mathbb C)$. Therefore, the corresponding incidence variety
$$
\{(A, x, y)| A\in \operatorname{PGL}(k+1,\mathbb C),\  (x,y)\in \Gamma_n,\  Ax=y  \}
$$ has the same dimension as $\operatorname{PGL}(k+1,\mathbb C)$, and hence its general fiber over $\operatorname{PGL}(k+1,\mathbb C)$ is zero-dimensional or empty.

By conservation of local intersection multiplicities, each isolated intersection of $\Gamma_n$ with $\Delta$ persists under a sufficiently small perturbation, with the same total multiplicity in a small neighborhood. Here we use $[\cdot]$ to denote the cohomology class. Since $[\Delta_A]=[\Delta]$, it follows that $\#P_n(f)\leq[\Gamma_n]\cdot[\Delta]$. Writing $\alpha=\pi_1^*[\omega]$ and $\beta=\pi_2^*[\omega]$, we have $[\Delta]=\sum_{j=0}^k\alpha^j\beta^{k-j}$. Pairing with $[\Gamma_n]$ gives
\[
    \#P_n(f)\leq\sum_{j=0}^k\int_{\Gamma_n}\pi_1^*(\omega^j)\wedge\pi_2^*(\omega^{k-j})=\sum_{j=0}^k\delta_j(f^n).
\]
By \eqref{eq:projective-stability}, $\delta_p(f^n)=\lambda^n$. Choose $\theta$ such that $\max_{j\ne p}d_j(f)<\theta<\lambda$. The definition of dynamical degrees gives $\delta_j(f^n)=o(\theta^n)$ for every $j\ne p$. Consequently,
\[
    \#P_n(f)\leq\lambda^n+o(\theta^n) =\lambda^n+o(\lambda^n).
\]
Thus \cref{thm:main} applies to both cases. 
\end{proof}

The canonical measures recalled in the preceding subsections are adapted and have entropy $\log\lambda$. Thus \cref{thm:main} resolves the open problems concerning their properties stated after \cref{prop:applications}.

\appendix
\section*{Appendix: Verification of (A1)--(A7)}
\phantomsection
\label[appendix]{app:estimates}

\setcounter{equation}{0}
\renewcommand{\theequation}{\arabic{equation}}
\renewcommand{\theHequation}{appendix.\arabic{equation}}

We verify conditions \textup{(A1)--(A7)} of \cite[\S3.1]{LOP26JEMS}. The same conditions are also used in \cite{ALP24MAMS}. So, we will use the results of both papers freely. When $\cS=\varnothing$, the results of \cref{thm:mme-coding} is true without \textup{(A1)--(A7)}. Therefore, we will assume $\cS\neq\varnothing$. Recall that $U=X\setminus\cS$ and $\rho(x)=\min\{1,\dist(x,\cS)\}$. The estimates in the proof of \cite[Lemma~7]{deThelin08Invent} give constants $C_0\geq1$ and $b>0$ such that
\begin{equation}\label{eq:derivative-growth}
    \|Df_x\|+\|\nabla Df_x\|+\|(Df_x)^{-1}\|\leq C_0\rho(x)^{-b},\qquad x\in U.
\end{equation}
Here $\nabla Df$ denotes the covariant derivative of $Df$ with respect to the fixed K\"ahler metric. 

After a constant rescaling of the metric, we may suppose that $\diam X\leq r_0<1$, so that $\rho(x)=\dist(x,\cS)\leq r_0$. Conditions \textup{(A1)--(A4)}, concerning exponential coordinates and their derivatives, are satisfied when the curvature tensor of $X$ is globally bounded; see the discussion in \cite[\S1.1]{ALP24MAMS} after condition \textup{(A4)} is introduced. Since $X$ is compact and smooth, this is automatically true. For simplicity, we omit these conditions here. 

It remains to verify \textup{(A5)--(A7)}. These assumptions are about the regularity of $f$. As mentioned in both \cite{LOP26JEMS,ALP24MAMS}, these conditions hold because of \eqref{eq:derivative-growth}. For completeness, we include here their full statements and the proofs.

Let $P_{u,v}\colon T_uX\to T_vX$ denote parallel transport along the unique minimizing geodesic in a convex normal neighborhood. For $x\in U$, let $g_x$ denote the local inverse branch of $f$ satisfying $g_x(f(x))=x$. To compare derivatives at different points, fix $x\in U$, for $y$ and $z$ in small convex normal neighborhood of $x$ and $f(x)$ respectively, we define
\begin{equation}\label{eq:parallel transport}
    \begin{aligned}
        \widetilde{Df}_y&:=P_{f(y),f(x)}\circ Df_y\circ P_{x,y}\colon T_xX\longrightarrow T_{f(x)}X\\
        (\widetilde{Dg_x})_z&:=P_{g_x(z),x}\circ D(g_x)_z\circ P_{f(x),z}\colon T_{f(x)}X\longrightarrow T_xX.
    \end{aligned}
\end{equation}

We now state \textup{(A5)--(A7)}. \medskip

\noindent\textbf{Regularity of $f$:} there exist constants $a>1$, $\beta>0$, and $K>1$ such that, for every $x$ such that both $x$, $f(x)\notin\cS$, there is $r(x)$ such that $\min\{\rho(x)^a,\rho(f(x))^a\}<r(x)<1$ for which, setting $D_x:=B(x,2r(x))$ and $E_x:=B(f(x),2r(x))$, the following hold.
\begin{enumerate}[label=\textup{(A\arabic*)},start=5]
    \item The restriction $f|_{D_x}$ is a diffeomorphism onto its image, and the inverse branch $g_x$ taking $f(x)$ to $x$ is a diffeomorphism from $E_x$ onto its image.

    \item The derivatives satisfy
    \[
        \rho(x)^a\leq\|Df_y\|\leq\rho(x)^{-a}\qquad(y\in D_x),
    \]
    and
    \[
        \rho(x)^a\leq\|D(g_x)_z\|\leq\rho(x)^{-a}\qquad(z\in E_x).
    \]

    \item The derivatives, identified by parallel transport as above, satisfy
    \[
        \|\widetilde{Df}_{y_1}-\widetilde{Df}_{y_2}\|\leq K\dist(y_1,y_2)^\beta\qquad(y_1,y_2\in D_x),
    \]
    and
    \[
        \|(\widetilde{Dg_x})_{z_1}-(\widetilde{Dg_x})_{z_2}\|\leq K\dist(z_1,z_2)^\beta\qquad(z_1,z_2\in E_x).
    \]
\end{enumerate}

\begin{proof}
For $x\in U$ and $y\in B(x,\rho(x)/2)$, by \eqref{eq:derivative-growth} we have 
\[
     \|Df_y\|+\|\nabla Df_y\|+\|(Df_y)^{-1}\|\leq C_0\rho(y)^{-b}\leq 2^bC_0 \rho(x)^{-b}.
\]
Therefore, the radius of a ball on which the inverse branch taking $f(x)$ to $x$ is defined can be bounded below by a positive power of $\rho(x)$.; see the proof of \cite[Lemma~2]{Briend-Duval99} and \cite{DD04JGA}. More precisely, we have constants $c>0$ and $N>1$, independent of $x$, such that $f$ is injective on $B(x,2c\rho(x)^N)$ and the inverse branch $g_x$ taking $f(x)$ to $x$ is defined on $B(f(x),2c\rho(x)^N)$, with image contained in $B(x,\rho(x)/2)$. Increasing $N$ and decreasing $c$ if necessary, we may also ensure that all the parallel transports in \eqref{eq:parallel transport} are well-defined in convex normal neighbourhoods.

Differentiating $f\circ g_x=\mathrm{id}$ gives $D(g_x)_z=(Df_{g_x(z)})^{-1}$ and
\[
    (\nabla D(g_x))_z(u,v)=-(Df_{g_x(z)})^{-1}(\nabla Df)_{g_x(z)}\bigl(D(g_x)_z u,D(g_x)_z v\bigr).
\]
Together with \eqref{eq:derivative-growth} and $\rho(g_x(z))\geq\rho(x)/2$, these identities give polynomial bounds for the first and second derivatives of $g_x$.

Consequently, after increasing $N$ and decreasing $c$ if necessary, there is a constant $C\geq1$, independent of $x$, such that, for any radius $0<r(x)\leq c\rho(x)^N$, the sets $D_x:=B(x,2r(x))$ and $E_x:=B(f(x),2r(x))$ satisfy
\[
\begin{aligned}
    C^{-1}\rho(x)^N&\leq\|Df_y\|\leq C\rho(x)^{-N} &&(y\in D_x),\\
    C^{-1}\rho(x)^N&\leq\|D(g_x)_z\|\leq C\rho(x)^{-N} &&(z\in E_x),
\end{aligned}
\]
and
\[
    \operatorname{Lip}_{D_x}(\widetilde{Df})+\operatorname{Lip}_{E_x}(\widetilde{Dg_x})\leq C\rho(x)^{-N}.
\]
The lower bounds of the first derivatives follow from $\|A\|\geq\|A^{-1}\|^{-1}$. For the Lipschitz bounds, differentiating \eqref{eq:parallel transport} and using the uniform bounds for the derivatives of parallel transport gives, after enlarging $C$ if necessary,
\[
\|d(\widetilde{Df})_y\|\le C\bigl(\|\nabla Df_y\|+\|Df_y\|+\|Df_y\|^2\bigr),
\]
where the last two terms arise from differentiating the parallel transports. Here $d(\widetilde{Df})_y$ denotes the differential at $y$ of the map $y\mapsto\widetilde{Df}_y$. By \eqref{eq:derivative-growth}, all these terms are bounded by negative powers of $\rho(x)$. Integrating along minimizing geodesics in $D_x$ gives the stated Lipschitz bound for $\widetilde{Df}$ after increasing $N$ if necessary. The same argument, using the bounds for $Dg_x$ and $\nabla Dg_x$ obtained above, gives the Lipschitz estimate for $\widetilde{Dg_x}$ on $E_x$.

Now suppose also that $f(x)\notin\cS$, and set $\rho_*(x):=\min\{\rho(x),\rho(f(x))\}$. Since $0<\rho(x)\leq r_0<1$, choose $a>2N$ sufficiently large that
\[
    2r_0^{a-N}\leq c,\qquad Cr_0^{a-N}\leq1,\qquad 2r_0^a<1.
\]
Taking $r(x):=2\rho_*(x)^a$, we obtain
\[
    \min\{\rho(x)^a,\rho(f(x))^a\}<r(x)<1,\qquad r(x)\leq c\rho(x)^N.
\]
Thus $f|_{D_x}$ and $g_x|_{E_x}$ are diffeomorphisms onto their images, proving \textup{(A5)}. The choice $Cr_0^{a-N}\leq1$ absorbs the constant $C$ in the two-sided derivative estimates and gives \textup{(A6)}; see also \cite[Remark~1.7]{ALP24MAMS}.

Finally, for $y_1,y_2\in D_x$, we have $\dist(y_1,y_2)\leq4r(x)=8\rho_*(x)^a$, and hence
\[
\begin{aligned}
    \|\widetilde{Df}_{y_1}-\widetilde{Df}_{y_2}\|
    &\leq C\rho(x)^{-N}\dist(y_1,y_2)\\
    &\leq \sqrt{8}\,C\rho(x)^{-N}\rho_*(x)^{a/2}\dist(y_1,y_2)^{1/2}\\
    &\leq \sqrt{8}\,C\dist(y_1,y_2)^{1/2}.
\end{aligned}
\]
The same argument gives
\[
    \|(\widetilde{Dg_x})_{z_1}-(\widetilde{Dg_x})_{z_2}\|\leq\sqrt{8}\,C\dist(z_1,z_2)^{1/2}\qquad(z_1,z_2\in E_x).
\]
Thus \textup{(A7)} holds with $\beta=1/2$ and $K=\sqrt{8}\,C$.
\end{proof}

\printbibliography[heading=bibliography,title={References}]

@article{Ovadia18JMD,
 author = {Ben Ovadia, Snir},
 title = {Symbolic dynamics for non-uniformly hyperbolic diffeomorphisms of compact smooth manifolds},
 fjournal = {Journal of Modern Dynamics},
 journal = {J. Mod. Dyn.},
 volume = {13},
 pages = {43--113},
 year = {2018},
}

@article{Guedj05ETDS,
 author = {Guedj, Vincent},
 title = {Topological entropy of meromorphic applications},
 fjournal = {Ergodic Theory and Dynamical Systems},
 journal = {Ergodic Theory Dyn. Syst.},
 volume = {25},
 number = {6},
 pages = {1847--1855},
 year = {2005},
}

@article{Bayraktar13JGA,
 author = {Bayraktar, Turgay},
 title = {Green currents for meromorphic maps of compact {K{\"a}hler} manifolds},
 fjournal = {The Journal of Geometric Analysis},
 journal = {J. Geom. Anal.},
 volume = {23},
 number = {2},
 pages = {970--998},
 year = {2013},
}

@article{BD01IHES,
  author        = {Briend, Jean-Yves and Duval, Julien},
  title         = {Deux caract\'erisations de la mesure d'\'equilibre d'un endomorphisme de {$\mathbb{P}^k(\mathbb{C})$}},
  journal       = {Publ. Math. Inst. Hautes {\'E}tudes Sci.},
  volume        = {93},
  year          = {2001},
  pages         = {145--159},
}

@article{Vigny14AIF,
  author        = {Vigny, Gabriel},
  title         = {Hyperbolic measure of maximal entropy for generic rational maps of {$\mathbb{P}^k$}},
  journal       = {Ann. Inst. Fourier (Grenoble)},
  volume        = {64},
  number        = {2},
  year          = {2014},
  pages         = {645--680},
}

@book{ALP24MAMS,
 author = {Araujo, Ermerson and Lima, Yuri and Poletti, Mauricio},
 title = {Symbolic dynamics for nonuniformly hyperbolic maps with singularities in high dimension},
 fseries = {Memoirs of the American Mathematical Society},
 series = {Mem. Am. Math. Soc.},
 volume = {1511},
 year = {2024},
 publisher = {Providence, RI: American Mathematical Society (AMS)},
}

@article{Buzzi20JMD,
 author = {Buzzi, J{\'e}r{\^o}me},
 title = {The degree of {Bowen} factors and injective codings of diffeomorphisms},
 fjournal = {Journal of Modern Dynamics},
 journal = {J. Mod. Dyn.},
 volume = {16},
 pages = {1--36},
 year = {2020},
}

@article{BCS22Annals,
 author = {Buzzi, J{\'e}r{\^o}me and Crovisier, Sylvain and Sarig, Omri},
 title = {Measures of maximal entropy for surface diffeomorphisms},
 fjournal = {Annals of Mathematics. Second Series},
 journal = {Ann. Math. (2)},
 volume = {195},
 number = {2},
 pages = {421--508},
 year = {2022},
}

@article{LOP26JEMS,
  author = {Lima, Yuri and Obata, Davi and Poletti, Mauricio},
  title = {Measures of maximal entropy for non-uniformly hyperbolic maps},
  journal = {J. Eur. Math. Soc.},
  year = {2026},
  note = {Published online first},
}

@article{Buzzi09ETDS,
 author = {Buzzi, J{\'e}r{\^o}me},
 title = {Maximal entropy measures for piecewise affine surface homeomorphisms},
 fjournal = {Ergodic Theory and Dynamical Systems},
 journal = {Ergodic Theory Dyn. Syst.},
 volume = {29},
 number = {6},
 pages = {1723--1763},
 year = {2009},
}

@article{Katok07JMD,
 author = {Katok, Anatole},
 title = {Fifty years of entropy in dynamics: 1958 -- 2007},
 fjournal = {Journal of Modern Dynamics},
 journal = {J. Mod. Dyn.},
 volume = {1},
 number = {4},
 pages = {545--596},
 year = {2007},
}

@article{Sarig13JAMS,
 author = {Sarig, Omri M.},
 title = {Symbolic dynamics for surface diffeomorphisms with positive entropy},
 fjournal = {Journal of the American Mathematical Society},
 journal = {J. Am. Math. Soc.},
 volume = {26},
 number = {2},
 pages = {341--426},
 year = {2013},
}

@article{DNT17BLMS,
  author = {Dinh, Tien-Cuong and Nguy{\^e}n, Vi{\^e}t-Anh and Truong, Tuyen Trung},
  title = {Growth of the number of periodic points for meromorphic maps},
  journal = {Bull. Lond. Math. Soc.},
  volume = {49},
  number = {6},
  pages = {947--964},
  year = {2017},
}

@article{deThelin08Invent,
  author = {{de Th{\'e}lin}, Henry},
  title = {Sur les exposants de {Lyapounov} des applications m{\'e}romorphes},
  journal = {Invent. Math.},
  volume = {172},
  number = {1},
  pages = {89--116},
  year = {2008},
  shorthand={dT08},
}

@article{Sarig11JMD,
  author = {Sarig, Omri M.},
  title = {{Bernoulli} equilibrium states for surface diffeomorphisms},
  journal = {J. Mod. Dyn.},
  volume = {5},
  number = {3},
  pages = {593--608},
  year = {2011},
}

@book {ABeBrDScSi10Book,
    AUTHOR = {Abate, Marco and Bedford, Eric and Brunella, Marco and Dinh,
              Tien-Cuong and Schleicher, Dierk and Sibony, Nessim},
     TITLE = {Holomorphic {D}ynamical {S}ystems},
    SERIES = {Lecture Notes in Mathematics},
    VOLUME = {1998},
 PUBLISHER = {Springer-Verlag, Berlin; Fondazione C.I.M.E., Florence},
      YEAR = {2010},
}

@article{DS05Annals,
  author        = {Dinh, Tien-Cuong and Sibony, Nessim},
  title         = {Une borne sup\'erieure pour l'entropie topologique d'une application rationnelle},
  journal       = {Ann. of Math. (2)},
  volume        = {161},
  number        = {3},
  year          = {2005},
  pages         = {1637--1644},
}

@article{Cantat01Acta,
  author        = {Cantat, Serge},
  title         = {Dynamique des automorphismes des surfaces {K3}},
  journal       = {Acta Math.},
  volume        = {187},
  number        = {1},
  year          = {2001},
  pages         = {1--57},
}

@article{Guedj05Annals,
 author = {Guedj, Vincent},
 title = {Ergodic properties of rational mappings with large topological degree},
 fjournal = {Annals of Mathematics. Second Series},
 journal = {Ann. Math. (2)},
 volume = {161},
 number = {3},
 pages = {1589--1607},
 year = {2005},
}

@article{DS05JFA,
  author        = {Dinh, Tien-Cuong and Sibony, Nessim},
  title         = {Dynamics of regular birational maps in {$\mathbb{P}^{k}$}},
  journal       = {J. Funct. Anal.},
  volume        = {222},
  number        = {1},
  year          = {2005},
  pages         = {202--216},
}

@article{BD26AM,
 author = {Bianchi, Fabrizio and Dinh, Tien-Cuong},
 title = {Exponential mixing of all orders and {CLT} for automorphisms of compact {K{\"a}hler} manifolds},
 fjournal = {Advances in Mathematics},
 journal = {Adv. Math.},
 volume = {485},
 pages = {21},
 note = {Id/No 110689},
 year = {2026},
}

@article{Briend-Duval99,
  AUTHOR = {Briend, Jean-Yves and Duval, Julien},
     TITLE = {Exposants de {L}iapounoff et distribution des points
              p\'eriodiques d'un endomorphisme de {$\mathbb {CP}^k$}},
   JOURNAL = {Acta Math.},
  FJOURNAL = {Acta Mathematica},
    VOLUME = {182},
      YEAR = {1999},
    NUMBER = {2},
     PAGES = {143--157},
}

@article{BLS93a,
 author = {Bedford, Eric and Lyubich, Mikhail and Smillie, John},
 title = {Polynomial diffeomorphisms of \texorpdfstring{$\mathbb {C}^2$}{c2}. {IV}: {The} measure of maximal entropy and laminar currents},
 fjournal = {Inventiones Mathematicae},
 journal = {Invent. Math.},
 volume = {112},
 number = {1},
 pages = {77--125},
 year = {1993},
}

@article{BLS93b,
 author = {Bedford, Eric and Lyubich, Mikhail and Smillie, John},
 title = {Distribution of periodic points of polynomial diffeomorphisms of \texorpdfstring{$\mathbb {C}^2$}{c2}},
 fjournal = {Inventiones Mathematicae},
 journal = {Invent. Math.},
 volume = {114},
 number = {2},
 pages = {277--288},
 year = {1993},
}

@article{BS91,
 author = {Bedford, Eric and Smillie, John},
 title = {Polynomial diffeomorphisms of \texorpdfstring{$\mathbb {C}^2$}{c2}: {Currents}, equilibrium measure and hyperbolicity},
 fjournal = {Inventiones Mathematicae},
 journal = {Invent. Math.},
 volume = {103},
 number = {1},
 pages = {69--99},
 year = {1991},
}

@article{BS92,
 author = {Bedford, Eric and Smillie, John},
 title = {Polynomial diffeomorphisms of \texorpdfstring{$\mathbb {C}^2$}{c2}. {III}: {Ergodicity}, exponents and entropy of the equilibrium measure},
 fjournal = {Mathematische Annalen},
 journal = {Math. Ann.},
 volume = {294},
 number = {3},
 pages = {395--420},
 year = {1992},
}

@book{dTH-Vigny,
 author = {{de Th{\'e}lin}, Henry and Vigny, Gabriel},
 title = {Entropy of meromorphic maps and dynamics of birational maps},
 fseries = {M{\'e}moires de la Soci{\'e}t{\'e} Math{\'e}matique de France. Nouvelle S{\'e}rie},
 series = {M{\'e}m. Soc. Math. Fr., Nouv. S{\'e}r.},
 volume = {122},
 year = {2010},
 publisher = {Soci{\'e}t{\'e} Math{\'e}matique de France (SMF), Paris},
  shorthand={dTV10},
}

@article{dTh10,
 author = {{de Th{\'e}lin}, Henry},
 title = {On the regular automorphisms of \texorpdfstring{$\mathbb {C}^k$}{ck}.},
 fjournal = {Publicacions Matem{\`a}tiques},
 journal = {Publ. Mat., Barc.},
 volume = {54},
 number = {1},
 pages = {243--262},
 year = {2010},
  shorthand={dT10},
}

@article{DF01AJM,
 author = {Diller, J. and Favre, C.},
 title = {Dynamics of bimeromorphic maps of surfaces},
 fjournal = {American Journal of Mathematics},
 journal = {Am. J. Math.},
 volume = {123},
 number = {6},
 pages = {1135--1169},
 year = {2001},
}

@article{DDG1Indiana,
 author = {Diller, Jeffrey and Dujardin, Romain and Guedj, Vincent},
 title = {Dynamics of meromorphic maps with small topological degree. {I}: from cohomology to currents},
 fjournal = {Indiana University Mathematics Journal},
 journal = {Indiana Univ. Math. J.},
 volume = {59},
 number = {2},
 pages = {521--562},
 year = {2010},
}

@article{DDG2CMH,
 author = {Diller, Jeffrey and Dujardin, Romain and Guedj, Vincent},
 title = {Dynamics of meromorphic mappings with small topological degree. {II}: {Energy} and invariant measure},
 fjournal = {Commentarii Mathematici Helvetici},
 journal = {Comment. Math. Helv.},
 volume = {86},
 number = {2},
 pages = {277--316},
 year = {2011},
}

@article{DDG3ENS,
 author = {Diller, Jeffrey and Dujardin, Romain and Guedj, Vincent},
 title = {Dynamics of meromorphic maps with small topological degree. {III}: {Geometric} currents and ergodic theory},
 fjournal = {Annales Scientifiques de l'{\'E}cole Normale Sup{\'e}rieure. Quatri{\`e}me S{\'e}rie},
 journal = {Ann. Sci. {\'E}c. Norm. Sup{\'e}r. (4)},
 volume = {43},
 number = {2},
 pages = {235--278},
 year = {2010},
}

@misc{DDK25,
    title={Exponential equidistribution of periodic points of endomorphisms of $\mathbb{P}^k$},
    author={{de Th\'elin}, Henry and Dinh, Tien-Cuong and Kaufmann, Lucas},
    year={2025},
    eprint={2409.19787},
    archivePrefix={arXiv},
     shorthand={dTDK25},
}

@misc{TDK26,
    title={Quantitative Dynamics of complex {H}\'enon maps},
    author={{de Th\'elin}, Henry and Dinh, Tien-Cuong and Kaufmann, Lucas},
    year={2026},
    eprint={2609.36105},
    shorthand={dTDK26},
    archivePrefix={arXiv},
}

@misc{Hu26,
    title={Degree growth of rational and meromorphic self-maps},
    author={Hu, Fei},
    year={2026},
    eprint={2610.18094},
    archivePrefix={arXiv},
}

@article{Vu21IMRN,
 author = {Vu, Duc-Viet},
 title = {Densities of currents on non-{K{\"a}hler} manifolds},
 fjournal = {IMRN. International Mathematics Research Notices},
 journal = {Int. Math. Res. Not.},
 volume = {2021},
 number = {17},
 pages = {13282--13304},
 year = {2021},
}

@article{dTH-Dinh12AM,
 author = {{de Th{\'e}lin}, Henry and Dinh, Tien-Cuong},
 title = {Dynamics of automorphisms on compact {K{\"a}hler} manifolds},
 fjournal = {Advances in Mathematics},
 journal = {Adv. Math.},
 volume = {229},
 number = {5},
 pages = {2640--2655},
 year = {2012},
 shorthand={dTD12},
}

@article{DNT15,
 author = {Tien-Cuong Dinh and Vi{\^e}t Anh Nguy{\^e}n and Tuyen Trung Truong},
 title = {Equidistribution for meromorphic maps with dominant topological degree},
 fjournal = {Indiana University Mathematics Journal},
 journal = {Indiana Univ. Math. J.},
 volume = {64},
 number = {6},
 pages = {1805--1828},
 year = {2015},
}

@article {DS05JAMS,
    AUTHOR = {Dinh, Tien-Cuong and Sibony, Nessim},
     TITLE = {Green currents for holomorphic automorphisms of compact
              {K}\"ahler manifolds},
   JOURNAL = {J. Amer. Math. Soc.},
  FJOURNAL = {Journal of the American Mathematical Society},
    VOLUME = {18},
      YEAR = {2005},
    NUMBER = {2},
     PAGES = {291--312},
}

@article {DS06CMH,
    AUTHOR = {Dinh, Tien-Cuong and Sibony, Nessim},
     TITLE = {Distribution des valeurs de transformations m\'eromorphes et
              applications},
   JOURNAL = {Comment. Math. Helv.},
  FJOURNAL = {Commentarii Mathematici Helvetici. A Journal of the Swiss
              Mathematical Society},
    VOLUME = {81},
      YEAR = {2006},
    NUMBER = {1},
     PAGES = {221--258},
}

@article {DS10JAG,
    AUTHOR = {Dinh, Tien-Cuong and Sibony, Nessim},
     TITLE = {Super-potentials for currents on compact {K}\"ahler manifolds
              and dynamics of automorphisms},
   JOURNAL = {J. Algebraic Geom.},
  FJOURNAL = {Journal of Algebraic Geometry},
    VOLUME = {19},
      YEAR = {2010},
    NUMBER = {3},
     PAGES = {473--529},
}

@article{DS16MAnn,
  author  = {Dinh, Tien-Cuong and Sibony, Nessim},
  title   = {Equidistribution of saddle periodic points for 
             {H\'enon}-type automorphisms of 
             {\texorpdfstring{$\mathbb{C}^k$}{C^k}}},
  fjournal = {Mathematische Annalen},
  journal  = {Math. Ann.},
  volume   = {366},
  number   = {3-4},
  pages    = {1207--1251},
  year     = {2016},
}

@incollection{Dujardin22,
 author = {Dujardin, Romain},
 title = {Geometric methods in holomorphic dynamics},
 booktitle = {International congress of mathematicians 2022, ICM 2022, Helsinki, Finland, virtual, July 6--14, 2022. Volume 5. Sections 9--11},
 pages = {3460--3482},
 year = {2023},
 publisher = {Berlin: European Mathematical Society (EMS)},
}

@article{Dujardin06duke,
 author = {Dujardin, Romain},
 title = {Laminar currents and birational dynamics},
 fjournal = {Duke Mathematical Journal},
 journal = {Duke Math. J.},
 volume = {131},
 number = {2},
 pages = {219--247},
 year = {2006},
}

@article{LZ25,
  author  = {Luo, Muhan and Zhou, Qi},
  title   = {Equidistribution of saddle periodic points for {H}{\'e}non-like maps},
  journal = {J. Eur. Math. Soc. (JEMS)},
  note    = {to appear},
}

@book{Kitchens,
 author = {Kitchens, Bruce P.},
 title = {Symbolic dynamics. {One}-sided, two-sided and countable state {Markov} shifts},
 fseries = {Universitext},
 series = {Universitext},
 year = {1998},
 publisher = {Berlin: Springer},
}

@book{Durrett99ESP,
 author = {Durrett, Rick},
 title = {Essentials of stochastic processes},
 fseries = {Springer Texts in Statistics},
 series = {Springer Texts Stat.},
 year = {1999},
 publisher = {New York, NY: Springer},
}

@article{Gurevich69,
 author = {Gurevich, B. M.},
 title = {Topological entropy of enumerable {Markov} chains},
 fjournal = {Soviet Mathematics. Doklady},
 journal = {Sov. Math., Dokl.},
 volume = {10},
 pages = {911--915},
 year = {1969},
}

@article{Gurevich70,
 author = {Gurevich, B. M.},
 title = {Shift entropy and {Markov} measures in the path space of a denumerable graph},
 fjournal = {Soviet Mathematics. Doklady},
 journal = {Sov. Math., Dokl.},
 volume = {11},
 pages = {744--747},
 year = {1970},
}

@misc{LZ26,
    title={Equidistribution of saddle periodic points for automorphisms on compact {K}\"ahler manifolds}, 
      author={Muhan Luo and Qi Zhou},
      year={2026},
      eprint={2609.30204},
      archivePrefix={arXiv},
}

@misc{DV24,
 author = {{de Th{\'e}lin}, Henry and Vigny, Gabriel},
 title = {Exponential mixing of all orders and {CLT} for generic birational maps of {$\mathbb{P}^k$}},
 year = {2024},
 eprint={2402.01178},
 archivePrefix={arXiv},
 shorthand={dTV24},
}

@incollection{Guedj10,
 author = {Guedj, Vincent},
 title = {Ergodic properties of rational maps},
 booktitle = {Quelques aspects des syst\`emes dynamiques polynomiaux},
 pages = {97--202},
 year = {2010},
 publisher = {Paris: Soci{\'e}t{\'e} Math{\'e}matique de France (SMF)},
}

@article{Lyubich83FAA,
 author = {Lyubich, M. Yu.},
 title = {The maximum-entropy measure of a rational endomorphism of the {Riemann} sphere},
 fjournal = {Functional Analysis and its Applications},
 journal = {Funct. Anal. Appl.},
 volume = {16},
 pages = {309--311},
 year = {1983},
}

@incollection {S99PS,
    AUTHOR = {Sibony, Nessim},
     TITLE = {Dynamique des applications rationnelles de {$\mathbb P^k$}},
 BOOKTITLE = {Dynamique et g\'eom\'etrie complexes ({L}yon, 1997)},
    SERIES = {Panor. Synth\`eses},
    VOLUME = {8},
     PAGES = {ix--x, xi--xii, 97--185},
 PUBLISHER = {Soc. Math. France, Paris},
      YEAR = {1999},
}

@article{BD05Duke,
 author = {Bedford, Eric and Diller, Jeffrey},
 title = {Energy and invariant measures for birational surface maps},
 fjournal = {Duke Mathematical Journal},
 journal = {Duke Math. J.},
 volume = {128},
 number = {2},
 pages = {331--368},
 year = {2005},
}

@article{DD04JGA,
 author = {Dinh, Tien-Cuong and Dupont, Christophe},
 title = {Dimension of the equilibrium measure of holomorphic maps},
 fjournal = {The Journal of Geometric Analysis},
 journal = {J. Geom. Anal.},
 volume = {14},
 number = {4},
 pages = {613--627},
 year = {2004},
}

@article {Rohl,
    AUTHOR = {Rohlin, V. A.},
     TITLE = {Exact endomorphisms of a {L}ebesgue space},
   JOURNAL = {Izv. Akad. Nauk SSSR Ser. Mat.},
  FJOURNAL = {Izvestiya Akademii Nauk SSSR. Seriya Matematicheskaya},
    VOLUME = {25},
      YEAR = {1961},
     PAGES = {499--530},
}

\end{document}